\documentclass[a4paper,reqno]{amsart}
\usepackage[utf8]{inputenc}
\usepackage{amssymb,amsthm,amsmath}
\usepackage[british]{babel}
\usepackage{xcolor}
\usepackage[all]{xy}
\usepackage{hyperref}

\theoremstyle{plain}
\newtheorem{theorem}[subsection]{Theorem}

\newtheorem{proposition}[subsection]{Proposition}

\theoremstyle{definition}

\theoremstyle{remark}
\newtheorem{example}[subsection]{Example}

\def\t{\otimes}
\newcommand{\gnab}{\text{\rm \textexclamdown}}

\DeclareMathOperator{\Hom}{Hom}
\DeclareMathOperator{\id}{id}
\DeclareMathOperator{\ev}{ev}

\newcommand{\C}{\ensuremath{\mathbb{C}}}

\newcommand{\Cb}{\ensuremath{\mathbf{C}}}
\newcommand{\Db}{\ensuremath{\mathbf{D}}}
\newcommand{\K}{\ensuremath{\mathbb{K}}}

\newcommand{\two}{\mathbf{2}}

\newcommand{\op}{\textnormal{op}}

\newcommand{\az}{\ensuremath{\alpha}}
\newcommand{\lz}{\ensuremath{\lambda}}
\newcommand{\rz}{\ensuremath{\rho}}
\newcommand{\sz}{\ensuremath{\sigma}}

\newcommand{\Vect}{\ensuremath{\mathsf{Vect}}}

\newcommand{\Pt}{\ensuremath{\mathsf{Pt}}}

\newcommand{\Set}{\ensuremath{\mathsf{Set}}}

\newcommand{\Mon}{\ensuremath{\mathsf{Mon}}}
\newcommand{\cMon}{\ensuremath{\mathsf{Mon^{c}}}}
\newcommand{\CoMon}{\ensuremath{\mathsf{CoMon}}}
\newcommand{\cCoMon}{\ensuremath{\mathsf{CoMon^{coc}}}}
\newcommand{\Lie}{\ensuremath{\mathsf{Lie}}}
\newcommand{\BiMon}{\ensuremath{\mathsf{BiMon}}}

\newcommand{\Monu}{\underline{\Mon}}
\newcommand{\cMonu}{\underline{\cMon}}
\newcommand{\CoMonu}{\underline{\CoMon}}
\newcommand{\cCoMonu}{\underline{\cCoMon}}
\newcommand{\MonC}{\Mon(\Cb)}
\newcommand{\cMonC}{\cMon(\Cb)}
\newcommand{\CoMonC}{\CoMon(\Cb)}
\newcommand{\cCoMonC}{\cCoMon(\Cb)}
\newcommand{\BiMonC}{\BiMon(\Cb)}

\newcommand{\LieC}{\Lie(\Cb)}
\newcommand{\MonAct}[2]{\ensuremath{\mathsf{MonAct}(#1,#2)}}
\newcommand{\LieAct}[2]{\ensuremath{\mathsf{LieAct}(#1,#2)}}
\newcommand{\MonActu}[2]{\ensuremath{\underline{\mathsf{MonAct}}(#1,#2)}}
\newcommand{\LieActu}[2]{\ensuremath{\underline{\mathsf{LieAct}}(#1,#2)}}
\newcommand{\LMonAct}[2]{\Lie(\MonActu{#1}{#2})}
\newcommand{\LLieAct}[2]{\Lie(\LieActu{#1}{#2})}

\newcommand{\LACC}{{\rm (LACC)}}
\newcommand{\Ker}{\textnormal{Ker}}

\def\pullback{
 \ar@{-}[]+R+<6pt,-1pt>;[]+RD+<6pt,-6pt>%
 \ar@{-}[]+D+<1pt,-6pt>;[]+RD+<6pt,-6pt>}

\def\dottedpullback{%
 \ar@{.}[]+R+<6pt,-1pt>;[]+RD+<6pt,-6pt>%
 \ar@{.}[]+D+<1pt,-6pt>;[]+RD+<6pt,-6pt>}

\newdir{>}{{}*:(1,-.2)@^{>}*:(1,+.2)@_{>}}
\newdir{<}{{}*:(1,+.2)@^{<}*:(1,-.2)@_{<}}

\begin{document}

\title[Algebraic exponentiation for Lie algebras]{Algebraic exponentiation for Lie algebras}

	\author{Xabier García-Martínez}
\address[Xabier García-Martínez]{Departamento de Matemáticas, Esc.\ Sup.\ de Enx.\ Informática, Campus de Ourense, Universidade de Vigo, E--32004, Ourense, Spain
	\newline and\newline
	Faculty of Engineering, Vrije Universiteit Brussel, Pleinlaan 2, B--1050 Brussel, Belgium}
\email{xabier.garcia.martinez@uvigo.gal}

\author{James R. A. Gray}
\address[James Gray]{Mathematics Division, Department of Mathematical Sciences Stellenbosch University, Private Bag X1,7602 Matieland, South Africa}
\email{jamesgray@sun.ac.za}


\thanks{The first author was supported by Ministerio de Economía y Competitividad (Spain), with grant number MTM2016-79661-P and is a Postdoctoral Fellow of the Research Foundation–Flanders (FWO).
The second author was supported in part by the National Research Foundation of South Africa (Grant Number 109203)}

\begin{abstract}
It is known that the category of Lie algebras over a ring admits \emph{algebraic exponents}. 
The aim of this paper is to show that the same is true for the category of internal Lie algebras in an 
additive, cocomplete, symmetric, closed, monoidal category. 
In this way, we add some new examples to the brief list of known locally algebraically cartesian closed categories, 
including the categories of Lie superalgebras and differentially graded Lie algebras amongst others. 
\end{abstract}


\maketitle
\section{Introduction}
Let $\C$ be a finitely complete category. Given an object $B$ of~$\C$, we write $\Pt_{B}(\C)$ for the \emph{category of points over $B$} in which an object $(A, p, s)$ is a split epimorphism~$p\colon {A\to B}$ in $\C$, together with a chosen section $s\colon {B\to A}$, so that $p\circ s=1_{B}$.
A morphism of $\Pt_B(\C)$ from $(A,p, s)$ to $(A', p', s')$ is a morphism $f \colon A \to A'$ in $\C$ such that $p' \circ f = p$ and $f \circ s = s'$. For each morphism $q \colon E \to B$, there is a \emph{change-of-base} functor $q^{*} \colon \Pt_B(\C) \to \Pt_E(\C)$, that takes a point $(A, p, s)$ to the point $(E\times_{B} A, \pi_1,\langle \id,sq\rangle)$ obtained from the pullback
\[
\xymatrix{E\times_{B}A \ar[r]^-{\pi_2} \ar@<0ex>[d]^-{\pi_1}  \pullback & A \ar@<.5ex>[d]^-{p}  \\
	E \ar[r]_-{q} \ar@{..>}@<1ex>[u]^-{\langle \id, sp\rangle} & B \ar@{..>}@<.5ex>[u]^-{s}}
\]

If for each morphism $q$ in $\C$, the change-of-base functor $q^{*}$ reflects isomorphisms, the category is said to be \emph{Bourn-protomodular} \cite{Bourn1991}. If $\C$ is pointed, this is equivalent to the same condition restricted to morphisms whose domain is the zero object. Note that the composition of the change-of-base functor $\gnab_B^{*} \colon \Pt_B(\C) \to \Pt_0(\C)$, where $\gnab_B \colon 0\to B$ is the unique morphism from the zero object to $B$, with the isomorphism $\Pt_0(\C) \cong \C$ is the kernel functor $\Ker_B$ sending each point $(A,p,s)$ to the kernel of $p$. Therefore, a finitely complete pointed category is Bourn-protomodular if and only if for each $B$ in $\C$ the kernel functor reflects isomorphisms, or equivalently, the split short five lemma holds. Recall that a pointed protomodular category which is Barr-exact and has finite coproducts is called a \emph{semi-abelian} category \cite{Janelidze-Marki-Tholen}.

If the change-of-base functor $q^{*}$ has a right adjoint for each morphism $q$, then~$\C$ is said to be \emph{locally algebraically cartesian closed} (\LACC{} for short) \cite{Bourn-Gray, GrayPhD, Gray2012}. 
If~$\C$ is also protomodular and pointed, then by \cite[Theorem 5.1]{Gray2012} it is sufficient to check if for each $B$ in $\C$ the kernel functor $\Ker_B \colon \Pt_B(\C) \to \C$ has a right adjoint. 

If~$\C$ is a \LACC, semi-abelian category, it implies several categorical-algebraic properties are satisfied by $\C$, such as \emph{peri-abelianness}~\cite{Bourn-Peri}, \emph{strong protomodularity}~\cite{B1}, the \emph{Smith is Huq} condition~\cite{MFVdL}, \emph{normality of Higgins commutators}~\cite{CGrayVdL1}, and \emph{algebraic coherence}. 

As we can see, \LACC{} is a strong condition where the short list of known semi-abelian examples includes groups, Lie algebras, crossed modules, cocommutative Hopf algebras over a field and all abelian categories. 
In fact, it was shown in~\cite{GM-Vdl-2} that \LACC{} characterizes Lie algebras amongst all varieties of non-associative algebras over an infinite field of characteristic not equal to two. Furthermore, for an infinite field of characteristic two there are exactly two non-abelian \LACC{} sub-varieties: Lie algebras and quasi-Lie algebras; that is the variety of algebras obtained from Lie algebras by replacing the identity $xx=0$ by $xy=-yx$.

It is not known whether there is any relation between \LACC\ and \emph{action representability}~\cite{BJK2}. They coincide in non-associative algebras over an infinite field (they both characterize Lie algebras~\cite{GTVV}), but it is not longer true over finite fields, since Boolean algebras are action representable but not \LACC~\cite{BJK2, Gray2012}.

The aim of the present paper is to show that, under certain conditions, the category of Lie algebras in a monoidal category is \LACC. In doing so, we will add the categories of Lie superalgebras, Lie colour algebras and differentially graded Lie algebras, amongst others, to the brief list of known \LACC{} examples. Another interesting example, is the category of Lie algebras in the Loday-Pirashvili category~\cite{Loday-Pirashvili}. In this way, we will show that the category of Leibniz algebras (which is not \LACC~\cite{GM-Vdl}) is a full reflective subcategory of a \LACC{} category.

Let us briefly explain our approach, which is motivated by \cite{GrayLie} where it is shown that the category of Lie algebras over a commutative ring is \LACC. Let $\Cb$ be a symmetric monoidal closed category with underlying category additive and small complete. 
We show that for an object $B$ in $\LieC$, the category of internal Lie algebras in $\Cb$, the category $\Pt_B(\LieC)$ is equivalent to the category $\Lie(\Db)$ of internal Lie algebras in the symmetric monoidal category $\Db$ with underlying category the category of \emph{internal monoid actions} of $U(B)$, the \emph{universal enveloping monoid} of~$B$, acting on the objects of $\Cb$, and with monoidal structure induced by $\Cb$ together with a natural comonoid structure on~$U(B)$. Moreover, we show that the kernel functor $\Ker_B \colon \Pt_B(\LieC)\to \LieC$ factors, via this isomorphism, through a forgetful functor $V\colon\Lie(\Db) \to \LieC$. This reduces the problem of finding a right adjoint of $\Ker_B$ to finding a right adjoint of $V$, which can be found by an adjoint lifting theorem.

The paper is organized as follows: In Section~\ref{algebras} after recalling the definition of various relevant types of internal algebras we construct adjoint functors between certain categories of such internal algebras, which we will use in Section~\ref{actions}. In Section~\ref{monoidal functors} we recall the necessary background and then prove the adjoint lifting theorem which we use to produce the right adjoint of the functor $V$ mentioned above. The adjoint functors previously defined are used in Section~\ref{actions} to produce the symmetric monoidal category $\Db$ as well as the equivalence of categories between~$\Pt_B(\LieC)$ and  $\Lie(\Db)$ mentioned before. Finally, in Section~\ref{examples} we study some applications of the obtained result.
	
\section{Internal algebras}\label{algebras}

Throughout the rest of paper we denote by $\Cb = (\C, \t, I, \az, \lz, \rz, \sz)$ a \emph{symmetric monoidal category}. We refer to \cite{MacLane} for basic knowledge about this topic. 
Recall that~$\Cb$ is said to be \emph{closed} if for each object $X$ in $\C$ the endofunctor $X \t - \colon \C\to \C$ is a left adjoint. A chosen right adjoint will be denoted by $(-)^{X}$ and the counit by~$\ev$. 
As the main non-cartesian example we have the symmetric monoidal category $\Vect_{\K}$ of vector spaces with the tensor product $\t_{\K}$.
The canonical isomorphism between $\Hom_{\Vect_{\K}}(X \t_{\K} Z, Y)$ and $\Hom_{\Vect_{\K}}(Z, \Hom_{\Vect_{\K}}(X, Y))$ gives us the right adjoint~$(-)^{X} = \Hom_{\Vect_{\K}}(X, -)$. In this particular case, the counit of the adjunction $\ev_Y \colon Y^{X} \t X \to Y$ maps $f \t x$ to the evaluation $f(x)$.

The setting we need in this manuscript is the following: $\Cb$ is a symmetric monoidal closed category with underlying category $\C$ additive and cocomplete. 

Recall that a \emph{monoid} $(A, m, u)$ in $\Cb$ is a triple, where $A$ is an object of $\C$, $m \colon A \t A \to A$ and $u \colon I \to A$ are morphisms of $\C$, such that the diagrams
\[
\xymatrix{
A \t (A \t A) \ar[rr]^-{\az} \ar[d]_-{\id \t m} &&  (A \t A) \t A \ar[d]^-{m\t \id}   & A \t I \ar[r]^-{\id \t u} \ar[rd]_-{\rz} & A \t A \ar[d]^-{m} & I \t A \ar[l]_-{u \t \id} \ar[ld]^-{\lz}  \\
A \t A \ar[r]^-{m} & A & A\t A \ar[l]_-{m} & &A &
}
\]
commute \cite{MacLane}. 
If in addition the diagram
\[
\xymatrix{
A \t A \ar[rr]^-{\sz} \ar[rd]_-{m} & & A \t A \ar[dl]^-{m} \\
& A &
}
\]
commutes, then it is said to be a \emph{commutative monoid}. The categories of such objects will be denoted by $\MonC$ and $\cMonC$, respectively. Recall that since~$\Cb$ is a symmetric monoidal category, given two monoids $(A,m,u)$ and $(A',m',u')$ in $\Cb$ the triple $(A\t A', (m\t m') i, (u\t u')\lz)$, where $i$ is the \emph{middle interchange isomorphism}, is a monoid which we denote by $(A,m,u)\t (A',m',u')$ (which is commutative whenever both $(A,m,u)$ and $(A',m',u')$ are). Moreover, with this definition of tensor product the symmetric monoidal structure of $\Cb$ lifts to both~$\MonC$ and~$\cMonC$. Let us write $\Monu{\Cb}$ and $\cMonu{\Cb}$ for the respective symmetric monoidal categories.
The duals of the notions of monoid and commutative monoid are \emph{comonoid} and \emph{cocommutative comonoid}, respectively, and the categories of such objects are denoted $\CoMonC$ and $\cCoMonC$, respectively. By duality the symmetric monoidal structure on $\Cb$ lifts to both $\CoMonC$ and $\cCoMonC$, producing the symmetric monoidal categories $\CoMonu{\Cb}$ and $\cCoMonu{\Cb}$, respectively.

A \emph{bimonoid} can be defined as a comonoid in the symmetric monoidal category~$\MonC$ (or equivalently as a monoid in~$\CoMonC$). We present a bimonoid as a quintuple $(A, m, u, d, e)$ where $(A, m, u)$ is an monoid, $(A, d, e)$ is a comonoid and $d\colon (A,m,e) \to (A,m,u)\t (A,m,u)$ and~$e\colon (A,m,u) \to (I,\lz,\id)$ are monoid morphisms, or equivalently $m$ and $u$ are comonoid morphisms. The category of bimonoids will be denoted by $\BiMonC$.
A bimonoid is called \emph{commutative} if the monoid part is commutative, and it is called \emph{cocommutative} if the comonoid part is cocommutative.

Under the conditions assumed on $\Cb$ it is well known that free monoids exist, in fact much milder assumptions are needed (see e.g.~\cite{Lack:free_monoids} and the references therein).  
\begin{proposition}\label{prop:free_monoid}
The forgetful functor $G\colon\MonC \to \C$ has a left adjoint which we will denote by $F$.
\end{proposition}

A \emph{Lie algebra} in $\Cb$ is a pair $(X, b)$ where $X$ is an object and  $b \colon X \t X \to X$ is a morphism in $\C$  making the diagrams
\begin{equation}
\label{dia:lie}
\vcenter{
\xymatrix{
X\t X \ar[d]_-{0} \ar[r]^-{\id + \sz} & X\t X \ar[dl]^-{b} & X\t (X\t X) \ar[rr]^-{\id + \sz\az + (\sz\az)^2} \ar[d]_-{0} && X\t (X\t X)\ar[d]^-{\id \t b}\\
X  & & X && X\t X\ar[ll]^-{b}
}
}
\end{equation}
commute.
Note that we can sum morphisms since we are assuming that $\C$ is additive. The category of Lie algebras in $\Cb$ will be denoted by $\LieC$.
As expected, when $\Cb$ is the category of vector spaces (with the usual monoidal structure) Lie algebras in $\Cb$ are usual Lie algebras.

Throughout this paper we will denote the direct sum of $A$ and $B$ by $A\oplus B$ with projections $\pi_1\colon A\oplus B\to A$ and $\pi_2\colon A\oplus B\to B$ and inclusions $\iota_1\colon A\to A\oplus B$ and $\iota_2\colon B\to B\oplus A$. For morphisms $f\colon W\to A$, $g\colon W\to B$ and $h\colon A\to C$ and~$i\colon B\to C$ we will denote by $\langle f,g \rangle \colon W\to A\oplus B$ and $[h,i]\colon A\oplus B\to C$ the unique morphisms with $\pi_1 \langle f,g\rangle =f$, $\pi_2 \langle f,g\rangle =g$, $[h,i]\iota_1=h$ and $[h,i]\iota_2=i$.

Given a monoid $(A,m,u)$ in $\Cb$, the pair $(A,m(\id-\sigma))$ is an object in $\LieC$ and this assignment determines the object map of a functor $L \colon \MonC \to \LieC$ which is the identity on morphisms (see e.g.~\cite{GrayLie}).
The following proposition is probably known, but we couldn't find a reference so we included a proof.
\begin{proposition}
The functor $L \colon \MonC \to \LieC$ has a left adjoint, which we denote by $U$.
\end{proposition}
\begin{proof}
Let us write $\eta$ for the unit of the adjunction $F \dashv G$ from Proposition~\ref{prop:free_monoid}.
Let $(B,b)$ be an object in $\LieC$ and $(A,m,u)$ an object in $\MonC$. For a morphism $f\colon B\to A$ in $\C$  let $\bar f$ be the corresponding monoid morphism from $F(B)=(T(B),m_B,u_B)$ to $(A,m,u)$ obtained via the adjunction $F \dashv G$, that is the unique monoid morphism such that $G(\bar f)\eta =f$. By considering the diagram 
\[
\xymatrix{
B^{\t 2}\ar[rr]^-{b} \ar[rd]^-{f ^{\t 2}} \ar[d]_-{\id - \sz} \ar@{}[ddrrr]|(0.55)*+[F]{1} && B\ar[dr]^-{\eta} \ar[dd]_-{f}&\\
B^{\t 2} \ar[rd]^-{f^{\t 2}}\ar[dd]_-{\eta \t \eta} &A^{\t 2}\ar[d]^-{\id -\sz} && T(B)\ar[dl]^-{\bar f}\\
 &A^{\t 2}\ar[r]^-{m} & A &\\
T(B)^{\t 2} \ar[ru]^-{\bar f ^{\t 2}} \ar[r]_-{m_B} &T(B)\ar[ru]_-{\bar f}&
}
\]
ones sees that the sub-diagram $\xybox{(0,0.2)*+[F]{1}}$ commutes if and only if the outer arrows commute.  This implies that $f$ is a morphism from $(B,b)$ to $L(A,m,u)$ in $\LieC$ if and only if the diagram
\[
\xymatrix@C=10ex{
B^{\t 2} \ar@<0.5ex>[r]^-{\eta b}\ar@<-0.5ex>[r]_-{m_B(\eta\t \eta)(\id-\sz)} & T(B) \ar[r]^-{\bar f} & A
}
\]
is a fork in $\C$. However, the previous diagram is a fork if and only if the diagram
\[
\xymatrix@C=10ex{
B\oplus B^{\t 2} \ar@<0.5ex>[r]^-{[\eta,\eta b]}\ar@<-0.5ex>[r]_-{[\eta,m_B(\eta\t \eta)(\id-\sz)]} & T(B) \ar[r]^-{\bar f} & A
}
\]
is a fork.
Therefore, writing $r, s\colon F(B\oplus B^{\t2})\to F(B)$ for the corresponding monoid morphisms induced by $[\eta,\eta b]$ and $[\eta,m_B(\eta\t \eta)(\id -\sz)]$, respectively, we see that previous diagram is a fork if and only if the diagram
\[
\xymatrix@C=10ex{
F(B\oplus B^{\t2})\ar@<0.5ex>[r]^-{r}\ar@<-0.5ex>[r]_-{s} & F(B) \ar[r]^-{\bar f} & (A,m,u)
}
\]
is a fork in $\MonC$. Since the tensor $\t$ preserves coequalizers in each argument, it is well-known that it preserves reflexive coequalizers, and the forgetful functor~${\MonC\to \C}$ creates reflexive coequalizers. Noting that the parallel pair of morphisms in the previous diagram is reflexive it easily follows that if~$q\colon F(B)\to U(B,b)$ is the coequalizer of $r$ and $s$ in $\MonC$, then $U(B,b)$ together with the morphism $q\eta$ from $(B,b)$ to $L(U(B,b))$ in $\LieC$ is the universal morphism from $(B,b)$ to $L$.
\end{proof}
Next we show that for any $B$ in $\LieC$ the monoid $U(B)$ is the underlying monoid of a cocommutative bimonoid (in fact of a cocommutative Hopf monoid)~$\widetilde U(B)$, and the assignment $B \mapsto \widetilde U(B)$ is the object map of a functor which is left adjoint of a functor $P$ which we construct below. This functor $P$ generalizes the construction of the primitive Lie algebra of a bialgebra.
 
Recall, as mentioned above, that since $\Cb$ is symmetric monoidal the monoidal structure lifts to $\MonC$.
\begin{proposition}
For each monoid $A=(A,m,u)$ the morphism $\delta_A \colon A\to A\t A$ defined by $\delta_A = (u\t \id)\lz^{-1} + (\id\t u)\rz^{-1}$ is a morphism from $L(A)$ to $L(A\t A)$ in~$\LieC$, natural in $A$.
\end{proposition}
\begin{proof}
Let $A=(A,m,u)$ be a monoid in $\Cb$ and let $\delta_A \colon A\to A\t A$ be the morphism~$\delta_A= l_A + r_A$ where $l_A=(u\t \id) \lz^{-1}$ and $r_A=(\id \t u)\rz^{-1}$. The naturality of~$\delta$ easily follows from the naturality of $l$ and $r$.
Therefore, it only remains to show that $\delta_A$ is a morphism from $L(A)$ to $L(A\t A)$. We have
\begin{align*}
(m\t m) i (l\t r) &= (m\t m) i ((u\t \id)\t (\id \t u) (\lz^{-1} \t \rz^{-1}) \\
&= (m\t m) ((u\t \id) \t (\id \t u)) i (\lz^{-1}\t \rz^{-1})\\
& = (\lz \t \rz) i (\lz^{-1}\t \rz^{-1})\\
&= \sz
\end{align*}
and 
\begin{align*}
(m\t m) i\sz (l\t r) &= (m\t m) i\sz ((u\t \id)\t (\id \t u) (\lz^{-1} \t \rz^{-1}) \\
&= (m\t m) ((\id \t u) \t (u \t \id)) i\sz (\lz^{-1}\t \rz^{-1})\\
& = (\rz \t \lz) i\sz (\lz^{-1}\t \rz^{-1})\\
&= \sz
\end{align*}
and hence
\[
(m\t m) i (\id -\sz) (l \t r)  =0
\]
and
\[
(m \t m) i (\id -\sz)(r\t l) =(m \t m) i (\sz -\id)\sz(r\t l)=-(m \t m) i (\id -\sz)(l\t r)\sz=0.
\]
On the other hand it is easy to check that $l_A$ and $r_A$ are monoid morphisms from $A$ to $A\t A$ and hence they are morphisms from $L(A)$ to $L(A\t A)$ in $\LieC$. Therefore
\begin{align*}
(m\t m) i (\id-\sz) (\delta_A\t \delta_A) &= (m\t m) i (\id-\sz) (l_A\t l_A)+(m\t m) i (\id-\sz)(l_A\t r_A)\\
&\hspace{-0.03cm}+ (m\t m) i (\id-\sz)(r_A\t l_A) + (m\t m) i (\id-\sz)(r_A\t r_A)\\
&= l_Am(\id-\sz)+0+0+r_Am(\id-\sz)\\
&=\delta_A m(\id-\sz)
\end{align*}
as desired.
\end{proof}
Recall that a bimonoid $(A,m,u,d,e)$ is called a \emph{Hopf monoid} if there is a morphism $s\colon A\to A$ making the diagram
\[
\xymatrix@C=3ex@R=4ex{
&&A\t A \ar[rr]^-{\id \t s} && A\t A\ar[drr]^-{m}\\
A \ar[drr]_-{d}\ar[urr]^-{d}\ar[rrr]^-{e} &&& I \ar[rrr]^-{u} &&& A \\
&&A\t A \ar[rr]^-{s \t \id} && A\t A\ar[urr]_-{m} 
}
\]
commute. Recall also that such a morphism $s$ is unique (whenever it exists) and is called the \emph{antipode} of the Hopf monoid $(A,m,u,d,e)$.
Note that Goyvaerts and Vercruysse have studied a generalization of the functor $P$, below,  in \cite{Goyvaerts_Vercruysse}.
\begin{proposition}
Let $H$ be the forgetful functor from $\BiMonC\to \MonC$.
For a  bimonoid $A$, with comultiplication $d_A\colon A\to A\t A$, the morphism denoted by~${j_A\colon P(A) \to L(H(A))}$ forming part of the equalizer diagram
\[
\xymatrix{
P(A) \ar[r]^-{j_A} & L(H(A)) \ar@<0.5ex>[r]^-{\delta_{H(A)}}\ar@<-0.5ex>[r]_-{L(d_A)} & L(H(A)\t H(A)),
}
\]
is the component of a natural transformation from a functor~$P\colon\BiMonC \to \LieC$ to the functor $LH$. Moreover, the functor $P$ is part of an adjunction with left adjoint $\widetilde{U}$ and unit $\widetilde{\nu}$ such that: $H\widetilde U=U$ and $(j\circ \widetilde U) \widetilde \nu=\nu$ where $\nu$ is the unit of the adjunction $U \dashv L$. Furthermore, for each object $X$ in $\LieC$ the bimonoid~$\widetilde U(X)$ is a cocommutative Hopf monoid.
\end{proposition}
\begin{proof}
It is easy to see that the definition of $P$, on objects above, is extended uniquely, making it into a functor in such a way that $j$ is a natural transformation~$P\to LH$.
Now let $\nu$ be the unit of the adjunction $U \dashv L$.
For an object $X$ in~$\LieC$ by adjunction the morphism $\delta_{U(X)}\nu_X \colon X \to L(U(X)\t U(X))$ determines a monoid morphism $d \colon U(X)\to U(X)\t U(X)$ such that $L(d) \nu_X = \delta_{U(X)} \nu_X$. Note that this means that $d\nu_X = \delta_X \nu_X = (u_{U(X)}\t \nu_X) \lz^{-1} + (\nu_X \t u_{U(X)})\rz^{-1}$ where~$u_{U(X)}$ is the unit of $U(X)$. In what follows we will drop subscripts when there is little risk of confusion. To show that $d$ is coassociative it is sufficient to show that $\az (\id\t d) d \nu = (d \t \id)d \nu$. 
A straightforward calculation shows that:
\begin{align*}
\az(\id \t d) d \nu 
&= ((u\t u) \t \nu)\az (\id \t \lz^{-1}) \lz^{-1} + ((u \t \nu) \t u)\az(\id \t \rz^{-1})\lz^{-1}\\
&\ \  + ((\nu \t u)\t u)\az(\id \t \lz^{-1})\rz^{-1}\\
\intertext{and}
(d \t \id) d \nu 
&= ((u\t u) \t \nu) (\lz^{-1}\t \id) \lz^{-1} + ((u \t \nu) \t u)(\lz^{-1}\t \id)\lz^{-1}\\
&\ \  + ((\nu \t u)\t u)(\rz^{-1}\t \id)\lz^{-1},
\end{align*}
and hence by coherence $\az (\id d) d \nu = (d \t \id)d \nu$. Trivially $d$ is cocommutative. Let~$e\colon U(X)\to I$ be the unique monoid morphism such that $L(e) \nu =0$. Since
\begin{align*}
(e \t \id) d \nu & = (e \t \id)((u\t \nu)\lz^{-1} + (\nu \t u)\rz^{-1})\\
&= (\id \t  \nu) \lz^{-1} + (0\t u)\rz^{-1}\\
&= \lz^{-1} \nu 
\end{align*}
it follows that $e$ is the counit of $U(X)$. Let us write $\widetilde U(X)$ for the bimonoid with underlying monoid $U(X)$ and comonoid structure as described above. Letting~$s\colon U(X)\to U(X)$ be the unique monoid morphism such that $L(s)\nu = -\nu$ we see that
\begin{align*}
m(s\t \id) d \nu &= m(s\t \id) \delta \nu \\
&= m(s\t \id)((u\t \nu)\lz^{-1} + (\nu\t u)\rz^{-1})\\
&= m((u\t \nu)\lz^{-1} - (\nu\t u) \rz^{-1})\\
&= m((u\t \id)\lz^{-1}\nu) - m((\id \t u) \rz^{-1}\nu)\\
&= \nu - \nu\\
&=0\\
&=e \nu
\end{align*}
and hence $\widetilde U(X)$ is a cocommutative Hopf monoid with antipode $s$. Supposing that~$A$ is a bimonoid we will show that the natural isomorphism $\hom(X,LH(A)) \cong \hom(U(X),H(A))$ induces a natural isomorphism $\hom(X,P(A))\cong \hom(\widetilde U(X),A)$. For a monoid morphism $f\colon U(X)\to H(A)$ since the outer arrows of diagram
\[
\xymatrix{
X \ar[r]^-{\nu}\ar[d]_-{\nu} & L(U(X) \ar[r]^-{L(f)}\ar[d]_-{L(d_{\widetilde U(X)})} \ar@{}[dr]|*+[F]{1}& L(H(A))\ar[d]^-{L(d_A)}\\
U(X) \ar[r]^-{\delta_{U(X)}}\ar[dr]_-{L(f)} & L(U(X)^{\t 2}) \ar[r]_-{L(f^{\t 2})} & L(H(A)^{\t 2})\\
& L(H(A)) \ar[ur]_-{\delta_{H(A)}} &
}
\]
commute if and only if the square $\xybox{(0,0.2)*+[F]{1}}$ commutes, it follows that $f$ preserves the comultiplication if and only if the morphism $L(f) \nu$ factors through $j_A \colon P(A)\to L(H(A))$. Now suppose that $f$ preserves comultiplication. Since in any bimonoid the identity $e = e m d$ holds it follows that
\begin{align*}
L(e f)\nu &= L(e) L(m) L(d) L(f)\nu\\ 
 &= L(e) L(f) L(m) L(d) \nu\\ 
 &= L(ef) L(m) \delta \nu\\ 
 &= L(ef) (\nu+\nu)\\ 
\end{align*}
and hence $L(e f)\nu =0$ which implies $e f =e$ making $f$ a bimonoid morphism. As mentioned above we have that $L(d)\nu_X=\delta_{U(X)}\nu_X$. This produces a unique morphism $\widetilde \nu_X \colon X \to P(\widetilde U(X))$ such that $j_{\widetilde U (X)} \widetilde \nu_X = \nu_X$ which one easily shows is the $X$ component of the unit of $P\dashv \widetilde U$. 
\end{proof}

\section{Monoidal functors and lifting of adjoints}\label{monoidal functors}

Given two monoidal categories $\Cb=(\C,\t,I,\az,\lz,\rz)$ and $\Cb'=(\C',\t',\az',\lz',\rz')$ a \emph{lax monoidal functor} from  $\Cb$ to $\Cb'$ is a triple $(F,\theta,\phi)$ where $F\colon\C\to \C'$ is a functor, $\theta\colon I'\to F(I)$ is a morphism in $\C'$, and $(\phi_{A,B} \colon F(A)\t' F(B) \to F(A\t B))_{A,B \in \C}$ is a natural transformation, such that for all $A$, $B$ and $C$ in $\C$ the diagrams
\[
\xymatrix{
F(A)\t'(F(B)\t'F(C)) \ar[r]^-{\az'}\ar[d]_-{\id \t' \phi} & (F(A)\t' F(B))\t' F(C) \ar[d]^-{\phi \t' \id}\\
F(A)\t' F(B\t C) \ar[d]_-{\phi} & (F(A\t B) \t'  F(C)\ar[d]^-{\phi}\\
F(A\t(B\t C)) \ar[r]_-{F(\az)} & F((A\t B)\t C) 
}
\]
\[
\xymatrix{
F(A) \ar[r]^-{\lz'}\ar[ddr]_-{F(\lz)} & I'\t' F(A)\ar[d]^-{\theta \t' \id} & F(A)\t' I'\ar[d]_-{\id \t' \theta}  & F(A)\ar[l]_-{\rz'}\ar[ddl]^-{F(\rz)}\\
& F(I)\t'F(A)\ar[d]^-{\phi} & F(A)\t F(I) \ar[d]_-{\phi} &\\
& F(I\t A) & F(A\t I)
}
\]
commute. If $\Cb$ and $\Cb'$ are symmetric monoidal categories with \emph{symmetry} isomorphisms $\sz$ and $\sz'$ respectively, then a lax monoidal functor $(F,\theta,\phi)$ between $\Cb$ and $\Cb'$ is called a lax symmetric monoidal functor as soon as the diagram
\[
\xymatrix{
F(A)\t' F(B) \ar[r]^-{\sz'}\ar[d]_-{\phi} & F(B)\t' F(A)\ar[d]^-{\phi}\\
F(A\t B) \ar[r]_-{F(\sz)} & F(B\t A)
}
\]
commutes, 
for every $A$ and $B$ in $\C$. A lax symmetric monoidal functor $(F,\theta,\phi)$ is called strong when $\theta$ and $\phi$ are isomorphisms, and strict when they are identity morphisms.
 
Recall also that (symmetric) monoidal structures on categories lift to both functor categories and products of categories (componentwise). In particular, if $\Cb$ is a (symmetric) monoidal category, then by $\Cb \times \Cb$ we will mean the (symmetric) monoidal category with underlying category $\C\times \C$, with tensor product defined by~$(A,B) \t (A',B') = (A\t A',B\t B')$, and with remaining structure defined componentwise. For a symmetric monoidal category $\Cb$ there is via the coherence theorem a unique natural isomorphism $i_{A,B,C,D} \colon (A\t B)\t (C\t D) \to (A\t C)\t (B\t D)$ built from $\az$, $\lz$, $\rz$ and $\sz$ which we will call (middle) interchange isomorphism. Note that this produces a strong symmetric monoidal functor $(\t, \lz, i) \colon \Cb \t \Cb \to \Cb$. 

It is well known that if $\Cb=(\C,\t,I,\az,\lz,\rz)$ and $\Cb'=(\C',\t',\az',\lz',\rz')$ are monoidal categories, and $(F,\theta,\phi)$ is a lax monoidal functor $\Cb \to \Cb'$, then the assignment $(A,m,u) \mapsto (F(A),F(m)\phi, F(u)\theta)$ is the object map of a functor ${\Mon(\Cb)\to \Mon(\Cb')}$ which sends a morphism $f$ to $F(f)$.
In a similar way, when~$\C$,~$\C'$ and~$F$ are also additive, the assignment $(B,b) \mapsto (F(B),F(b)\phi)$ is the object map of a functor $\Lie(\Cb) \to \Lie(\Cb')$ which sends a morphism $f$ to $F(f)$. Let us denote this induced functor by $\Lie(F)$. The following fact may be known, but we couldn't find a refence:
\begin{proposition}\label{prop:adjunction_lifts_to_lie}
Let $\Cb=(\C,\t,I,\az,\lz,\rz)$ and $\Cb'=(\C',\t',\az',\lz',\rz')$ be monoidal categories such that $\C$ and $\C'$ are additive, and let $F\colon \Cb \to \Cb'$ be a strict monoidal functor which is additive. If $F\colon\C\to \C'$ has a right adjoint, then $\Lie(F)$ has a right adjoint. 
\end{proposition}
\begin{proof}
Suppose $F$ has a right adjoint $G$, and the unit and counit of the associated adjunction are denoted by $\eta$ and $\epsilon$, respectively. According to~\cite{Kelly:doctrinal} we know that defining $\theta=\eta_I$,
 and for each $A'$ and $B'$ in $\C'$, $\phi_{A',B'} \colon  G(A')\t G(B') \to G(A'\t'B')$ to be the composite $G(\epsilon_{A'}\t'\epsilon_{B'}) \eta_{G(A')\t G(B')}$
produces a lax monoidal functor $(G,\theta,\phi) \colon \Cb'\to \Cb$, and furthermore, these data make the triangle in \eqref{diag:right adjoint lifts}, below, commute. As explained above this lax functor determines a functor $\Lie(G) \colon \Lie(\Cb')\to \Lie(\Cb)$. The claim now follows by observing that for objects $(X,b)$ and $(X',b')$ in $\LieC$ and $\Lie(\Cb')$, respectively, and for $f\colon X\to G(X')$ a morphism in $\C$, either (and hence both) of the statements: $f \colon (X,b)\to (G(X),G(b')\phi)$ is a morphism in $\LieC$; and $\epsilon F(f)\colon (F(X),F(b)) \to (X',b')$ is a morphism in $\Lie(\Cb')$; are equivalent to the commutativity of the diagram
\begin{equation}
\label{diag:right adjoint lifts}
\vcenter{
\xymatrix@C=12ex{
 F(X)\t' F(X)\ar@{=}[d] \ar[r]^-{F(f)\t' F(f)} & FG(X')\t' FG(X') \ar[ddr]^-{\epsilon_{X'}\t'\epsilon_{X'}}\ar@{=}[d] & \\
F(X\t X) \ar[r]^-{F(f\t f)}\ar[dd]_-{F(b)} & F(G(X')\t G(X'))\ar[d]^-{F(\phi_{X',X'})} & \\
& FG(X'\t' X') \ar[r]^-{\epsilon_{X'\t' X'}}\ar[d]_-{FG(b')} & X'\t'X'\ar[d]^-{b'}\\
F(X)\ar[r]_-{F(f)} & FG(X') \ar[r]_-{\epsilon_{X'}} & X'.
}
}
\end{equation}
\end{proof} 

\section{Actions}\label{actions}

Recall that if $A=(A,m,u)$ is a monoid in $\Cb$ and $X$ is an object in $\C$ an action of $A$ on $X$ is a morphism $\phi \colon A\t X \to X$ making the diagrams
\[
\xymatrix{
(A\t A) \t X \ar[rr]^-{\az^{-1}} \ar[d]_-{m \t \id} && A \t (A \t X) \ar[d]^-{\id \t \phi}\\
A \t X \ar[r]_-{\phi} &X &A\t X\ar[l]^-{\phi}
}
\xymatrix{
& I\t X \ar[dr]_-{\lz} \ar[r]^-{u \t \id} & A \t X \ar[d]^-{\phi}\\
& & X
}
\]
commute.
Given an object $B=(B,b)$ in $\LieC$ and an object $X$ in $\C$ we say that a morphism $\theta \colon B\t X \to X$ is an action of $B$ on $X$ if
the diagram
\begin{equation}\label{eq:lie_act}
\vcenter{
\xymatrix{
(B\t B)\t X \ar[rr]^-{\az^{-1}((\id - \sz)\t \id)} \ar[d]_-{b\t \id} && B\t (B\t X) \ar[d]^-{\id \t \theta}\\
B\t X \ar[r]_-{\theta} &X &B\t X\ar[l]^-{\theta}
}
}
\end{equation}
commutes.
Let us write $E(X)$ for the monoid with underlying object $X^X$ and with multiplication $m \colon X^X\t X^X\to X^X$ and $u \colon I \to X^X$ the unique morphisms making the diagrams
\[
\xymatrix{
(X^X\t X^X) \t X \ar[rr]^-{\az^{-1}} \ar[d]_-{m \t \id} && X^X \t (X^X \t X) \ar[d]^-{\id \t \ev}\\
X^X \t X \ar[r]_-{\ev} &X &X^X\t X\ar[l]^-{\ev}
}
\xymatrix{
& I\t X \ar[dr]_-{\lz} \ar[r]^-{u \t \id} & X^X \t X \ar[d]^-{\ev}\\
& & X
}
\]
commute. Let $f \colon A' \to A$ be a morphism in $\MonC$, $g\colon B'\to B$ be a morphism in~$\LieC$, $\phi\colon A\t X\to X$ be an action of $A$ on $X$, and $\theta \colon B\t X\to X$ be an action of $B$ on $X$, then $\phi(f\t \id)$ is an action of $A'$ on $X$ and $\theta(g\t \id)$ is an action of $B'$ on $X$. We have
\begin{proposition}
For an object $X$ in $\C$ there are natural isomorphisms between the following functors:
\begin{enumerate}
\item the functor assigning to each $B$ in $\LieC^{\op}$ the set of actions of $B$ on $X$;
\item  $\hom(-,L(E(X))\colon\LieC^{\op}\to \Set$;
\item $\hom(U(-),E(X))\colon\LieC^{\op} \to \Set$;
\item the functor assigning to each $B$ in $\LieC^{\op}$ the set of monoid actions of $U(B)$ on $X$.
\end{enumerate}
 For an object $B$ in $\LieC$, under these natural isomorphisms,
an action $\theta$ of $B$ on $X$, a morphism $\bar \theta\colon B\to L(E(X))$ in $\LieC$, a morphism $\bar \phi\colon U(B)\to E(X)$ in~$\MonC$, and a monoid action $\phi$ of $U(B)$ on $X$ are related when:
$\theta = \ev (\bar \theta \t \id)$, $\bar \theta = \bar \phi \nu$, $\phi = \ev(\bar \phi \t \id)$,
and (hence) $\theta = \phi(\nu \t \id)$, where $\nu$ is the unit of the adjunction $U \dashv L$.
\end{proposition}
\begin{proof}
The natural isomorphism between the functors in (3) and (4) is standard, while the natural isomorphism between the functors in (2) and (3) is obtained from the adjunction $U \dashv L$. Given a morphism $\theta \colon B\t X\to X$, let $\bar \theta \colon B\to X^X$ be the unique morphism such that $\ev (\bar \theta \t \id)=\theta$. One easily observes that if the map~$\bar \theta \colon B\to L(E(X))$ is a morphism in $\LieC$, then the diagram $\xybox{(0,0.2)*+[F]{1}}$ below is commutative, making the whole diagram
\[
\xymatrix@C=7ex@R=4ex{
(B\t B)\t X \ar[r]^-{(\id-\sz)\t\id}\ar[rd]_-{(\bar \theta\t\bar \theta)\t \id}\ar[ddd]_-{b\t \id}\ar@{}[dddr]|*+[F]{1}&(B\t B)\t X\ar[r]^-{\az^{-1}}&B\t(B\t X)\ar[dd]_-{\bar \theta \t (\bar \theta \t \id)}\ar[rdd]^-{\id \t \theta} &\\
&(X^X\t X^X)\t X \ar[d]^-{(\id-\sz)\t \id}&&\\
&(X^X\t X^X)\t X\ar[d]^-{m \t \id}\ar[r]^-{\az^{-1}}&X^X\t(X^X\t X)\ar[d]^-{\id \t \ev}&B\t X\ar[ldd]^-{\theta}\\
B\t X \ar[r]^-{\bar \theta \t \id}\ar@/_2ex/[rrd]_-{\theta}&X^X\t X\ar[rd]^-{\ev}&X^X\t X\ar[d]^-{\ev}&\\
&&X&
}
\]
commute and hence $\theta$ is an action of $B$ on $X$. The converse follows immediately from the universal property of $\ev$. 
\end{proof}
Let $\two$ denote the category with two objects $0$ and $1$ and one non-identity morphism $0\to 1$. Recall that the functor category $\C^{\two}$ can be made into a symmetric monoidal category $\Cb^\two$ with monoidal structure defined componentwise. Recall also, that when $\C$ has pullbacks the symmetric monoidal category $\Cb^\two$ is monoidal closed. Since the functor $L \colon \Lie(\Cb^\two) \to \Mon(\Cb^\two)$ is, up to isomorphism, the same as the functor $L^\two \colon \LieC^{\two} \to \MonC^\two$, it follows from the previous proposition applied to $\Cb^\two$ that for a morphism $g\colon B\to B'$ in $\LieC$, a morphism $f\colon X\to X'$ in $\C$, and monoid actions $\phi$ and $\phi'$ of $U(B)$ on $X$ and $U(B')$ on $X'$ respectively, the outer arrows of the diagram
\[
\xymatrix{
B\t X \ar[r]^-{\nu \t \id}\ar[d]_-{g\t f} & U(B)\t X \ar[d]^-{U(g)\t f} \ar[r]^-{\phi} & X\ar[d]^-{f}\\
B'\t X' \ar[r]_-{\nu \t \id} & U(B')\t X' \ar[r]_-{\phi'} & X'
}
\]
commute if and only if the right hand square commutes. In particular, writing~$\LieAct{B}{\Cb}$ and $\MonAct{A}{\Cb}$ for the categories of $B$ actions and $A$ actions respectively (for some objects $B$ in $\LieC$ and $A$ in $\MonC$), this implies:
\begin{proposition}\label{prop:mon_act_isomorphic_lie_act}
For an object $B$ in $\LieC$ the assignment $(X,\phi) \mapsto (X,\phi(\nu \t \id))$ is the object map of an isomorphism $\MonAct{U(B)}{\Cb} \to \LieAct{B}{\Cb}$ which is identity on morphisms.
\end{proposition}
Now suppose that $A$ is a bimonoid with comultiplication $d\colon A\to A\t A$. If $\phi$ and $\phi'$ are actions of the underlying monoid of $A$ on $X$ and $Y$, respectively, then $\phi \t^{d} \phi' \colon A\t (X\t Y) \to (X\t Y)$ defined by $\phi \t^{d} \phi' = (\phi \t \phi') i (d \t \id)$ is a monoid action.
To see why, note that the monoidal functor $(\t,\lz,i) \colon \Cb\times \Cb \to \Cb$ sends the action $(\phi,\phi')$ of the monoid~$(A,A)$ on $(X,Y)$ to the action $(\phi \t \phi') i$ of the monoid~$A\t A$ on $X\t Y$. However, since $d \colon A\to A\t A$ is a monoid morphism it follows that $\phi \t^{d} \phi' = (\phi \t \phi')i(d\t \id)$ is an action of $A$ on $X\t Y$.

Noting that the counit of the bimonoid $A$ determines a monoid action of $A$ on $I$ one easily establishes that the symmetric monoidal structure of $\Cb$ lifts to $\MonAct{A}{\Cb}$ and hence the forgetful functor $G_A\colon \MonAct{A}{\Cb} \to \Cb$ is strict monoidal. Let us write~$\MonActu{A}{\Cb}$ for this monoidal category. It is well known that when $\Cb$ is monoidal closed $G_A$ has a right adjoint. The object map of this right adjoint assigns to each object $X$ the object $X^A$ with action $\phi \colon A\t X^A \to X^A$, the unique morphism making the diagram
\[
\xymatrix{
(A\t X^A) \t A \ar[rr]^-{\phi \t \id}\ar[d]_-{(\id \t \sz)\az(\sz\t \id)} && X^A\t A \ar[d]^-{\ev}\\
X^A\t (A\t A) \ar[r]_-{\id \t m} & X^A \t A \ar[r]^-{\ev} & X
}
\]
commute.
Therefore, since the categories $\C$ and $\MonAct{A}{\Cb}$ are additive and so is the functor $G_A$, by Proposition~\ref{prop:adjunction_lifts_to_lie}, it follows that: 
\begin{proposition}\label{prop:forget_lmonact_right_adj}
The induced forgetful functor $\Lie(G_A)\colon \LMonAct{A}{\Cb} \to \LieC$ has a right adjoint.
\end{proposition}
Suppose $B$ is an object in $\LieC$. Let us write $\LieActu{B}{\Cb}$ for the monoidal category obtained by translating the monoidal structure from $\MonActu{U(B)}{\Cb}$ via the isomorphism given in Proposition~\ref{prop:mon_act_isomorphic_lie_act}. 
Trivially: 
\begin{proposition}\label{prop:LLieAct_iso_LMonAct}
For an object $B$ in $\LieC$ the assignment given by $((X,\phi),b) \mapsto ((X,\phi(\nu\t \id)),b)$ is the object map of an isomorphism $\LMonAct{U(B)}{\Cb} \to \LLieAct{B}{\Cb}$ which is the identity on morphisms.
\end{proposition}

Let us calculate explicitly what the tensor product in $\LieActu{B}{\Cb}$ is. Suppose that $(X,\theta)$ and $(X',\theta')$ are objects in $\LieAct{B}{\Cb}$. Letting $(X,\phi)$ and $(X',\phi')$ be the corresponding objects in $\MonAct{U(B)}{\Cb}$ we see that the diagram
\[
\xymatrix@C=-4ex@R=3.5ex{
&&U(B)^{\t 2}\t (X\t X')\ar[dd]^-{i}&&\\
&(B\t I)\t (X\t X')\ar[ru]^-(0.4){(\nu\t u)\t \id}\ar[dd]^-{i}&&(I\t B)\t (X\t X')\ar[lu]_-(0.4){((u\t \nu)\t \id}\ar[dd]^-{i}&\\
B\t (X\t X')\ar[ru]^-{\rz^{-1}\t \id}\ar@/_17.22pt/[rddd]_-{\az}&&(U(B)\t X)\t(U(B)\t X')\ar[dd]^-{\phi\t \phi'}&&B\t (X \t X') \ar[lu]_-{\lz^{-1}\t \id}\ar@/^17.22pt/[lddd]^-{\sz \az (\id \t \sz)}\\
&(B\t X)\t (I\t X')\ar[dd]^-{\id \t \lz}\ar[ru]|{(\nu\t \id)\t (u \t \id)}&&(I\t X)\t(B\t X')\ar[lu]|{(u\t \id)\t(\nu \t \id)}\ar[dd]^-{\lz\t \id}&\\
&&X\t X'&&\\
&(B\t X)\t X'\ar[ru]_-{\theta \t \id}&&X\t(B\t X')\ar[lu]^-{\id \t \theta'}&
}
\]
commutes. This implies that
\begin{align*}
(\phi \t^{d} \phi') (\nu \t \id) &= (\phi \t \phi') i (d \nu \t \id)\\
&=(\phi \t \phi') i (\delta \nu \t \id) \\
&= (\phi \t \phi') i (((u\t \nu)\lz^{-1} + (\nu\t u)\rz^{-1}) \t \id)\\
&= (\theta \t \id) \az + (\id \t \theta') \sz \az (\id \t \sz)
\end{align*}
and hence the translated tensor product of $\LieActu{B}{\Cb}$ is defined by
\[
(X,\theta)\t (X',\theta') = (X\t X', \theta * \theta')
\]
where $\theta * \theta'= (\theta \t \id) \az + (\id \t \theta') \sz \az (\id \t \sz)$.
One easily observes that this means that an object in $\LLieAct{B}{\Cb}$ can be identified with a pair $((X,b), \theta)$ where~$(X,b)$ is an object in $\LieC$, and $\theta$ is an action of $B$ on $X$ making the diagram
\begin{equation}\label{eq:lie_lie_act}
\vcenter{
\xymatrix{
B\t (X\t X) \ar[r]^-{\id \t b} \ar[d]_-{\theta * \theta} & B\t X\ar[d]^-{\theta}\\
X\t X \ar[r]_-{b} & X
}
}
\end{equation}
commute.
Writing $\overline{G_B} \colon \LieActu{B}{\Cb} \to \C$ for the forgetful functor we have:
\begin{proposition}
For an object $B$ in $\LieC$ the map assigning to each object $(A,p,s)$ in $\Pt_B(\LieC)$ the pair $(X,\theta)$ where $X$ is the domain of $k\colon X\to A$ the kernel of $p$ and $\theta \colon B\t X \to X$ is the unique morphism in $\C$ with $k \theta = b_A (s\t k)$, is the object map of a functor $W\colon \Pt_B(\LieC) \to \LLieAct{B}{\Cb}$ which is part of an equivalence of categories
making the diagram
\begin{equation}
\label{diag:equiv_pts_llieact}
\vcenter{
\xymatrix@C=1ex{
\Pt_B(\LieC) \ar[rr]^-{W} \ar[dr]_-{\Ker_B} && \LLieAct{B}{\Cb} \ar[dl]^-{\Lie(\overline{G_B})} \\
& \LieC
}
}
\end{equation}
commute.
\end{proposition}
\begin{proof}
Let $B$ be an object in $\LieC$, for an object $(A,p,s)$ in $\Pt_B(\LieC)$, defining $W(A,p,s)=(X,\theta)$ where $X$ and $\theta$ are defined as above. One can check that the commutativity of diagrams \eqref{eq:lie_act} and \eqref{eq:lie_lie_act} follow from  commutativity of diagrams \eqref{dia:lie} and coherence. 

If $f\colon (A,p,s)\to (A',p',s')$ is a morphism in $\Pt_B(\LieC)$, then  an easy calculation shows that the induced map between the kernels of $p$ and $p'$ lifts to a morphism $W(A,p,s)\to W(A',p',s')$ producing a functor $W\colon\Pt_B(\LieC)\to \LLieAct{B}{\Cb}$ making diagram~\eqref{diag:equiv_pts_llieact} commute. On the other hand, given $(X,\theta)$ in $\LLieAct{B}{\Cb}$ letting  $b \colon (B\oplus X)\t (B\oplus X) \to B\oplus X$ be the unique morphism with $b(\iota_1\t \iota_1) = \iota_1 b_B$, $b(\iota_1\t\iota_2)=\theta$, $b(\iota_2\t \iota_1)=-\theta \sz$ and $b(\iota_2\t \iota_2)= \iota_2 b_X$ makes $((B\oplus X,b),\pi_1,\iota_1)$ into an object of $\Pt_B(\LieC)$ such that $W((B\oplus X,b),\pi_2,\iota_2) = (X,\theta)$. Therefore, noting that $W$ is faithful, to show that it is part of an equivalence of categories we need only show it is full. 

To that end suppose that $(A,p,s)$ and $(A',p',s')$
are objects in $\Pt_B(\LieC)$, $k\colon X\to A$ and $k'\colon X'\to A'$ are the kernels of
$p$ and $p'$, respectively, $\theta\colon B\t X\to X$ and $\theta'\colon B\t X'\to X'$
are morphisms such that $W(A,p,s)=(X,\theta)$ and $W(A',p',s')=(X',\theta')$,
and $f$ is a morphism from $W(A,p,s)$ to $W(A',p',s')$. Letting $l\colon A\to X$ be the unique morphism
in $\C$ such that $kl=\id-sp$, using the fact that $k$ and $s$ are jointly epimorphic,
one can easily check that if $g = k'fl+s'p$ then $g$ is a morphism from $(A,p,s)$ to $(A',p',s')$ such that $W(g)=f$.
\end{proof}
Combining the previous proposition with Proposition~\ref{prop:LLieAct_iso_LMonAct} we obtain:
\begin{proposition}
For each $B$ in $\LieC$ the kernel functor $\Ker_B \colon \Pt_B(\LieC) \to \C$ factors through the forgetful functor $\Lie(G_{U(B)}) \colon \Lie(\MonActu{U(B)}{\Cb}) \to \LieC$ via an equivalence of categories.
\end{proposition}
We are now ready to prove our main theorem:
\begin{theorem}\label{theorem:main}
For a symmetric monoidal closed category $\Cb = (\C,\t,I,\az,\lz,\rz,\sz)$ such that $\C$ is cocomplete and additive, the category $\LieC$ is \LACC.
\end{theorem}
\begin{proof}
Combining the previous proposition and Proposition~\ref{prop:forget_lmonact_right_adj} we see that for each~$B$ in $\LieC$ the functor $\Ker_B \colon \Pt_B(\LieC) \to \LieC$ has a right adjoint. The claim now follows by \cite[Theorem 5.1]{Gray2012}.
\end{proof}

\section{Examples}\label{examples}
\begin{example}
Let $R$ be a commutative ring and consider the category of chain complexes of $R$-modules. If $(V, d)$ and $(V', d')$ are chain complexes, the tensor product of $(V,d)$ and $(V',d')$ is the chain complex $(V\t V', \delta)$ where
\[
(V \t V')_i = \bigoplus_{j + k = i}V_j \t V'_k,
\]
and $\delta$ is defined for all $v \in V_j$ and $v' \in V'_k$ by
\[
\delta(v \t v') = d(v) \t v' + (-1)^{j} v \t d'(v').
\]
The symmetry isomorphism $\sz \colon V \t V' \to V' \t V$ is defined for all $v \in V_j$ and $v' \in V'_k$ by $\sz(v \t v') = (-1)^{ij} v' \t v$. 
With this tensor product, symmetry isomorphism and the remaining structure defined canonically, it forms a symmetric monoidal closed category whose underlying category is cocomplete and abelian.

Its category of internal Lie algebras is equivalent to the category whose objects are \emph{differential graded Lie algebras}, i.e., a $\mathbb{Z}$-graded $R$-module $V$, with a linear map $d \colon V \to V$ of degree $-1$ such that $d \circ d = 0$, and a bilinear map $[-, -] \colon V \t V \to V$ of degree zero satisfying: for all homogeneus $x, y, z$ in $V$ 
\begin{enumerate}
	\item $[x, y] = -(-1)^{|x||y|} [y, x]$,
	\item $(-1)^{|x||z|} [x, [y, z]] + (-1)^{|y||x|} [y, [z, x]] + (-1)^{|z||y|} [z, [x, y]]$,
	\item $d([x, y]) = [d(x), y] + (-1)^{|x|}[x, d(y)]$
\end{enumerate}
where $|v|$ denotes the degree of $v$.
By Theorem~\ref{theorem:main} it follows that the category of differential graded Lie algebras is \LACC.
\end{example}

\begin{example}
If in the previous example we consider the full subcategory formed by chain complexes where $d = 0$, internal Lie algebras are essentially \emph{Lie colour algebras}. 
Applying Theorem~\ref{theorem:main} again we know that the category of Lie colour algebras is \LACC.
	
In a similar way, the category of Lie superalgebras over a commutative ring can be seen as the category of internal Lie algebras in the symmetric monoidal closed category of \emph{super $R$-modules} and hence it is \LACC.
\end{example}

\begin{example}
Let us consider now the category whose objects are linear maps $f \colon V \to W$ and morphisms are homomorphisms of $R$-modules making the diagram commutative
\[
\xymatrix{
V \ar[r] \ar[d]_-{f} & V' \ar[d]^-{f'} \\
W \ar[r] & W'
}
\]
There is a tensor product, called the \emph{infinitesimal tensor product} defined as
\[
\left(\vcenter{\xymatrix{ V \ar[d]^-{f} \\ W }}\right) \t \left(\vcenter{\xymatrix{ V' \ar[d]^-{f'} \\ W' }}\right) =
\left(\vcenter{\xymatrix{ (V \t W') \oplus (W \t V') \ar[d]^-{[f \t \id, \id \t f']} \\ W \t W' }}\right)
\]

This monoidal category was introduced in \cite{Loday-Pirashvili} as \emph{the tensor category of linear maps}, and in more recent papers was renamed to the \emph{Loday-Pirashvili category}~\cite{FCGMLa, GMLa, Ro}, denoted by $\mathcal{LP}$. An internal Lie algebra in this category is a linear map $f \colon M \to \mathfrak{g}$, where~$\mathfrak{g}$ is a Lie algebra, $M$ is a right $\mathfrak{g}$-module and $f$ preserves the $\mathfrak{g}$-action. Once more, as a consequence of Theorem~\ref{theorem:main}, the category of Lie algebras in $\mathcal{LP}$ is \LACC. Furthermore, the category of Leibniz algebras can be found as a full reflective subcategory of internal Lie algebras in $\mathcal{LP}$, but it is known that the category of Leibniz algebras is not \LACC~\cite{GM-Vdl}. 
\end{example}

\begin{example}
	Whereas the categories of \emph{Hom-Lie algebras} and \emph{multiplicative Hom-Lie algebras} are known to be semi-abelian categories satisfying very few categorical-algebraic properties (see~\cite{CaGM}), it was shown in~\cite{GoVe} that \emph{regular Hom-Lie algebras} can be found as internal categories in a certain symmetric monoidal category, and therefore, again by Theorem~\ref{theorem:main}, it is \LACC.
\end{example}

\section*{Acknowledgements}
The first author is grateful to Stellenbosch University for their kind hospitality during his stay there.


\end{document}